%% file: main.tex
\documentclass{amsart}

\input{preamble.tex}

\title{
    Knot quandles distinguish Suciu's ribbon knots
}

\author{
    Jumpei Yasuda
}

\date{
    \today
}

\address{Department of Mathematics, Graduate School of Science, Osaka Metropolitan University, Osaka 558‐8585, JAPAN}
\email{j.yasuda@omu.ac.jp}


\begin{document}
    \maketitle
    \input{abstract.tex}

\input{Introduction.tex}
    \input{quandle.tex}
    \input{ribbon.tex}
    \input{acknowlegements}

    \bibliographystyle{plain}
    \bibliography{reference.bib}
\end{document}

%% file: preamble.tex
\usepackage[margin=30truemm]{geometry}
\usepackage[utf8]{inputenc}

\usepackage{amsthm}
\usepackage{amsmath}
\usepackage{amsfonts}
\usepackage{amssymb}
\usepackage{amscd}
\usepackage{mathrsfs}
\usepackage{tikz}
\usetikzlibrary{cd}
\usepackage{ascmac}
\usepackage{multirow}
\usepackage{graphicx}

\usepackage{mathtools}
\usepackage{cleveref}
\usepackage{url}
\usepackage{tikz-cd}
\usepackage{enumitem}
\usepackage{subcaption}
\usepackage{afterpage}
\usepackage{changepage}
\usepackage{bm}

\theoremstyle{plain}
    
    \newtheorem{theorem}{Theorem}[section]
    \newtheorem{proposition}[theorem]{Proposition}

    \newtheorem{claim}[theorem]{Claim}

\theoremstyle{definition}
    \newtheorem{definition}[theorem]{Definition}

    \newtheorem{question}[theorem]{Question}

\theoremstyle{remark}

\crefname{maintheorem}{Theorem}{Theorems}
\crefname{claim}{Claim}{Claims}

\newcommand{\C}{\mathbb{C}}
\newcommand{\R}{\mathbb{R}}
\newcommand{\Z}{\mathbb{Z}}

\newcommand{\SL}{\mathop{\mathrm{SL}}\nolimits}

\newcommand{\GL}{\mathop{\mathrm{GL}}\nolimits}

\newcommand{\As}{\mathop{\mathrm{As}}\nolimits}

%% file: abstract.tex
\begin{abstract}
    The knot quandle is an invariant of $n$-knots.
    In this note, we study the knot quandles of Suciu’s ribbon $n$-knots, an infinite family of knots with isomorphic knot groups.
    We prove that their knot quandles are mutually non-isomorphic.
    Furthermore, we compute types of these quandles.
\end{abstract}

%% file: Introduction.tex
\section{Introduction}

An \textit{n-knot} $K$ is an oriented $n$-sphere $S^n$ smoothly embedded in $S^{n+2}$ ($n \geq 1$), and the \textit{knot group} of $K$, denoted by $G(K)$, is the fundamental group of the complement.

A \textit{quandle} \cite{Joyce82,Matveev84} is a non-empty set $Q$ equipped with a binary operation $*$ satisfying the following:
\begin{enumerate}
    \item For any $x \in Q$, $x*x = x$.
    \item For any $y \in Q$, the map $S_y: Q \to Q$ defined by $S_y(x) = x*y$ is a bijection.
    \item For any $x,y,z \in Q$, $(x*y)*z = (x*z)*(y*z)$.
\end{enumerate}
These axioms are related to Reidemeister moves in knot theory.
Two quandles $Q$ and $Q'$ are \textit{isomorphic} if there exists a bijection $f: Q \to Q'$ such that $f(x*y) = f(x)*f(y)$ for all $x,y \in Q$.

For an $n$-knot, the \textit{knot quandle} is defined as an invariant (\Cref{Def of knot quandle}).
This is a powerful invariant of 1-knots.
In fact, the knot quandle is a complete invariant of 1-knots up to orientations of knots and $S^3$.
For higher-dimensional knots, little is known about the strength of knot quandles.

The knot group is obtained from the knot quandle via the \textit{associated group}.
Thus, if knot quandles are isomorphic,then the corresponding knot groups are also isomorphic.
This leads to the following question:
\begin{question}
    Are knot quandles stronger invariants of $n$-knots than knot groups?
\end{question}

In \cite{Tanaka-Taniguchi2025}, Tanaka and Taniguchi gave the first examples answering this question for 2-knots:
They provided infinitely many triples of 2-knots $(K_1, K_2, K_3)$ whose knot groups are isomorphic but whose knot quandles are mutually non-isomorphic.
To distinguish their knot quandles, they proved that the \textit{types} of these quandles are different (\Cref{Def of type}).

In \cite{Suciu1985}, Suciu discovered infinitely many distinct (ribbon) $n$-knots, say $R_k^n$, whose knot groups are all isomorphic, where $k \geq 1$ and $n \geq 2$ are integers.
Definitions are given in \Cref{Section: ribbon}.
For $n=2$, he used $\pi_2(S^4 \setminus R_k^2)$ as a $\Z\pi_1$-module to distinguish them.

In this note, we focus on the knot quandles of Suciu's ribbon knots and prove the following theorem:

\begin{theorem}\label{Main Theorem}
    For each $n \geq 2$, the knot quandles of $R_k^n$, $k \geq 1$, are mutually non-isomorphic.
\end{theorem}

By \Cref{Main Theorem}, we obtain an infinite family of $n$-knots answering the above question for general $n \geq 2$.
Furthermore, by the following theorem, the knot quandles of Suciu's ribbon knots cannot be distinguished by their types.

\begin{theorem}\label{Type of suciu ribbon knots}
    For each $n \geq 2$ and $k \geq 2$, the type of the knot quandle of $R_k^n$ is equal to $\infty$.
\end{theorem}

This note is organized as follows.
In \Cref{Section: quandle}, we review the definitions of knot quandles, associated groups, and types of quandles.
In \Cref{Section: ribbon}, we review Suciu's ribbon knots and prove \Cref{Main Theorem} and \Cref{Type of suciu ribbon knots}.

%% file: quandle.tex
\section{Knot quandles of $n$-knots}
\label{Section: quandle}

In this section, we introduce the knot quandles of $n$-knots ($n \geq 1$) and review the associated groups and types of quandles.

Let $K$ be an $n$-knot ($n \geq 1$), $N_K$ a tubular neighborhood of $K$, and $E_K = S^{n+2} \setminus N_K$ the exterior of $K$.
We fix a base point $p \in E_K$.
A \textit{meridional disk} of $K$ is a disk $D$ properly embedded in $N_K$ intersecting $K$ transversely at a single point, and a \textit{noose} of $K$ is a pair of a meridional disk $D$ of $K$ and a path $\alpha$ in $E_K$ from a point of $\partial D$ to $p$.
For a noose $(D, \alpha)$, the loop $\alpha^{-1}\, \partial D\, \alpha$ is called a \textit{meridional loop} of $K$, where $\partial D$ is a loop starting at $\alpha(0)$ and going once along $\partial D$ in the positive direction.
The set of homotopy classes of all nooses of $K$ is denoted by $Q(K, p)$.

\begin{definition}[\cite{Joyce82,Matveev84}]\label{Def of knot quandle}
    The \textit{knot quandle} of $K$ is $Q(K, p)$ equipped with the operation defined by $[(D, \alpha)]*[(D',\beta)] ~=~ [(D, \alpha \cdot \beta^{-1} \partial D' \beta)]$.
    Since the isomorphism class of $Q(K,p)$ is independent of the choice of $p$, we simply write $Q(K)$.
\end{definition}


For a quandle $Q$ and an integer $r \geq 0$, the \textit{associated $r$-group}, denoted by $\As_r(Q)$, is the group with the presentation
\begin{align*}
    \langle x \in Q \mid x*y = y^{-1}xy, x^r = 1 ~(x,y \in Q) \rangle.
\end{align*}
In particular, the associated $0$-group is simply called the \textit{associated group} of a quandle $Q$ and denoted by $\As(Q)$.

To prove \Cref{Main Theorem}, we define the subgroup $P_r(Q)$ of $\As_r(Q)$.

\begin{definition}[cf. \cite{Winker1984}]\label{Def of PnQ}
    Let $\psi_r: \As_r(Q) \to \As_r(Q)^{\mathrm{ab}}$ be the abelianization map.
    The subgroup $P_r(Q)$ of $\As_r(Q)$ is defined as the kernel of $\psi_r$.
\end{definition}

The map $\eta: Q(K) \to G(K)$ sending $[(D, \alpha)] \in Q(K)$ to $[\alpha^{-1}\partial D \alpha] \in G(K)$ induces a group isomorphism $\As(Q(K)) \to G(K)$.
We identify $\As(Q(K))$ with $G(K)$ by this isomorphism.
Since all meridional loops of $K$ are conjugate to each other in $G(K)$, the preimage of each meridional loop of $K$ by $\eta$ is non-empty set.

In \cite[Theorem 5.2.2]{Winker1984}, Winker proved the following proposition for the case of 1-knots.
Here, we prove for general $n$-knots ($n\geq 1$).

\begin{proposition}\label{Result of Winker}
    Let $\Sigma_r(K)$ be the $r$-fold cyclic cover of $S^{n+2}$ branched along an $n$-knot $K$ ($r \geq 2$).
    Then, $P_r(Q(K))$ is group isomorphic to $\pi_1(\Sigma_r(K))$.
\end{proposition}

\begin{proof}
    We write $Q = Q(K)$, $G = \As(Q)$, and $H = \As_r(Q)$ for simplicity.
    Let $\psi: G \to \Z/r\Z$ be the composition of the abelianization map $G \to \Z$ and the quotient map $\Z \to \Z/r\Z$ and $\psi_r: H \to \Z/r\Z$ the abelianization map.
    By the definition, we have $\mathrm{Ker}(\psi_r) = P_r(Q)$.

    Let $M_r(K)$ be the $r$-fold cyclic cover of $S^{n+2} \setminus K$.
    It is known that $\pi_1(M_r(K)) = \mathrm{Ker}(\psi)$ and $\pi_1(\Sigma_r(K)) = \pi_1(M_r(K))/ \langle\langle \mu^r \rangle\rangle$, where $\mu \in G$ is a meridional loop of $K$.

    Since $H$ is obtained from $G$ by adding relations $x^r = 1$ for all $x \in Q$, the map $\pi: G \to H$ defined by $\pi(x) = x$ is a surjective group homomorphism.
    In addition, $\ker(\pi) = \langle\langle \mu^r \rangle\rangle$ holds.
    Thus, we obtain the following commutative diagram:
    \[
        \begin{tikzcd}
            & 1 \arrow[d] & 1 \arrow[d] & 1 \arrow[d] \\[-2mm]
            1 \arrow[r] & \langle\langle \mu^r \rangle\rangle\arrow[r, equal]\arrow[d, "f"] & \langle\langle \mu^r \rangle\rangle\arrow[d, "\ker(\pi)"]\arrow[r] & 1\arrow[d]\arrow[r] & 1\\
            1 \arrow[r] & \pi_1(M_r(K)) \arrow[r, "\ker(\psi)"]\arrow[d, "g"] & G\arrow[r, "\psi"]\arrow[d, "\pi"] & \Z/r\Z\arrow[d, equal]\arrow[r] & 1\\
            1 \arrow[r] & P_r(Q) \arrow[r, "\ker(\psi_r)"]\arrow[d] & H\arrow[r, "\psi_r"]\arrow[d] & \Z/r\Z\arrow[r]\arrow[d] & 1\\[-2mm]
            & 1 & 1 & 1
        \end{tikzcd}
    \]
    Here, group homomorphisms $f$ and $g$ are obtained by universality of $\ker(\psi)$ and $\ker(\psi_r)$, respectively.
    Hence, by the nine lemma, we have the short exact sequence 
    \begin{align*}
        1 \to \langle\langle \mu^r \rangle\rangle \to \pi_1(M_r(K)) \to P_r(Q(K)) \to 1.
    \end{align*}
    Therefore, $P_r(Q)$ is isomorphic to $\pi_1(M_r(K))/ \langle\langle \mu^r \rangle\rangle \cong \pi_1(\Sigma_r(K))$.
\end{proof}

\begin{definition}\label{Def of type}
    The \textit{type} of a quandle $Q$ is the positive integer (or $\infty$) defined by
    \begin{align*}
        \mathrm{type}(Q) ~=~ \min \{i \in \Z_{\geq 1} \mid x*^i y = x\mbox{ for any $x, y \in Q$}\} \cup \{\infty\},
    \end{align*}
    where $x*^i y = S_y^n(x)$.
    Here, $\mathrm{type}(Q) = \infty$ means there is a pair of elements $x,y \in Q$ such that $x*^i y \neq x$ for any $i \geq 1$.
\end{definition}

In \cite{Tanaka-Taniguchi2025}, Tanaka and Taniguchi showed that the type of the knot quandle of a fibered 2-knot is equal to the order of the monodromy of the fibration.

%% file: ribbon.tex
\section{Suciu's ribbon knots}\label{Section: ribbon}

We recall the notions of Suciu's ribbon $n$-knots $R_k^n$.
In the following two paragraphs, we identify $S^{n+2} = \R^{n+2} \cup \{\infty\}$ and $\R^{n+2} = \R^3 \times \R^{n-1}$.

Let $n \geq 1$ be an integer and $D^n$ a $n$-disk.
An \textit{$n$-link} $L$ is a union of disjoint $n$-knots, and an $n$-link is called \textit{trivial} if it is a union of boundaries of disjoint $(n+1)$-balls embedded in $S^{n+2}$.
A \textit{1-handle} $h$ attached to an $n$-link $L$ is the image of an embedding $h: D^{n}\times [0,1] \to S^{n+2}$ such that $h(D^n\times [0,1]) \cap L = h(D^n \times \{0,1\})$.
A 1-handle attached to a 1-link is also called a \textit{band}.
An $n$-knot $K$ is called \textit{ribbon} if it is obtained from a trivial $n$-link $L_0$ by surgery along disjoint 1-handles $h_1, \dots, h_m$:
\begin{align*}
    K ~=~ L_0 \cup \partial H \setminus \mathrm{int}(H \cap L_0), \quad H ~=~ \bigcup_{i=1}^m h_i(D^n \times [0,1]).
\end{align*}

Let $l_0$ be a 3-component trivial 1-link and $d_0$ a union of disjoint disks in $\R^3$ with $\partial d_0 = l_0$.
For each $k \geq 1$, let $b_k$ be the union of disjoint two bands attached to $l_0$ as shown in \Cref{Fig: diagram of banded links}.
Fix an integer $n \geq 2$.
Set $L_0 = \partial (d_0 \times [-2,2]^{n-1})$ and $H_k = b_k \times [-1,1]^{n-1}$.
Then, $L_0$ is the $3$-component trivial $n$-link in $\R^{n+2}$ and $H_k$ is the union of disjoint two 1-handles attached to $L_0$.
The result of surgery on $L_0$ along $H_k$, denoted by $R_k^n$, is an $n$-sphere.
We call $R_k^n$ \textit{Suciu's ribbon $n$-knots} \cite{Suciu1985}.
\begin{figure}[h]
    \centering
    \includegraphics[width = \hsize]{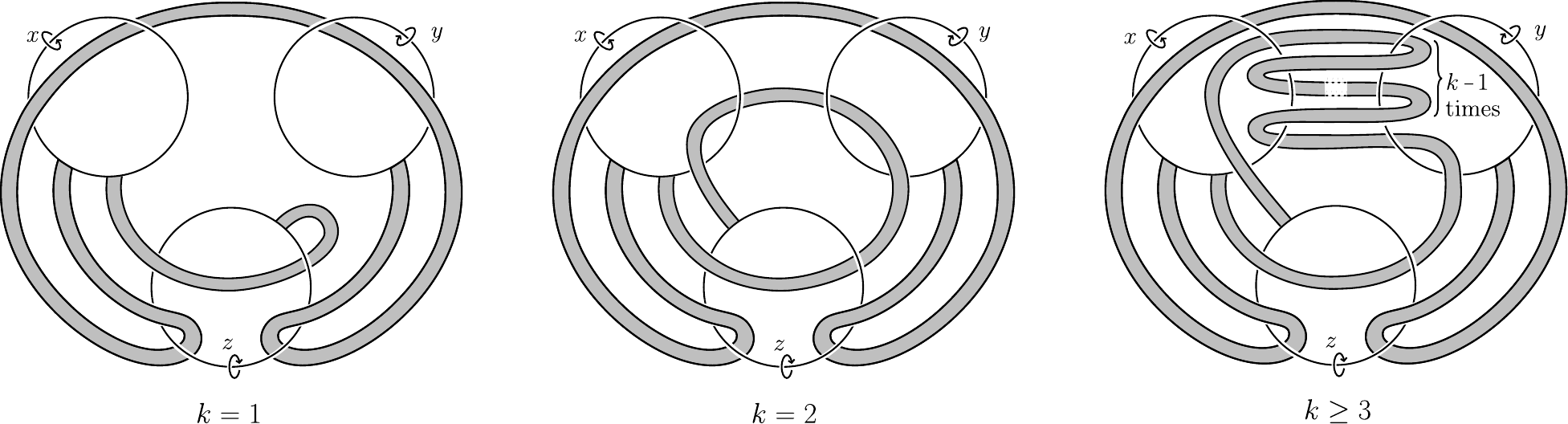}
    \caption{Diagrams of $l_0$ and $b_k$.}
    \label{Fig: diagram of banded links}
\end{figure}

Since $R_k^n$ is obtained from $L_0$ by surgery along 1-handles, $G(R_k^n)$ is generated by meridional loops of $L_0$ with two relations associated with two 1-handles.
Thus, we have the following presentation of $G(R_k^n)$:
\begin{align*}
    G_k ~=~ \langle x,y,z \mid x = y^V,~ x = z^W \rangle,
\end{align*}
where $x, y, z$ are represented by meridional loops of $L_0$ as shown in \Cref{Fig: diagram of banded links}, $V = zyx^{-1}z^{-1}$, $W = (xy^{-1})^{k-1}z^{-1}$, and $x^V = V^{-1}xV$.
Each relation is obtained by reading crossings between $l_k$ and $b_k$ along each band of $b_k$ (\cite{Yajima1969}).

Let $B_3$ be the braid group on three strands with the presentation $\langle \sigma, \tau \mid \sigma\tau\sigma =  \tau\sigma\tau \rangle$.

\begin{proposition}[\cite{Kanenobu-Sumi2020-JKTR, Suciu1985}]\label{Isomorphisms of knot groups}
    For each $k \geq 1$, $G_k$ is isomorphic to $B_3$ by the isomorphism sending
    \begin{align*}
        x \mapsto \tau^{-1}\sigma\tau(\tau^{-1}\sigma)^{k-1}, \quad y \mapsto \sigma(\tau^{-1}\sigma)^{k-1}, \quad z \mapsto (\tau^{-1}\sigma)^k\tau.
    \end{align*}
\end{proposition}

\begin{proof}
    It follows from the explicit Tietze transformations between $G_k$ and $B_3$ constructed in \cite[Proposition 2.1]{Kanenobu-Sumi2020-JKTR}.
\end{proof}

As mentioned in the proof of \Cref{Result of Winker}, the fundamental group $\pi_1(\Sigma_2(R_k^n))$ is isomorphic to $\mathrm{Ker}(\psi)/\langle \langle x^2 \rangle \rangle$, where $\psi: G_k \to \Z/2\Z$ is the composition of the abelianization map of $G_k$ and the quotient map $\Z \to \Z/2\Z$.
In \cite[Proposition 2.2]{Kanenobu-Sumi2020-JKTR}, Kanenobu and Sumi gave the following presentation of $\mathrm{Ker}(\psi)/\langle \langle x^2 \rangle \rangle$ by the Reidemeister-Schreier method:
\begin{align*}
    S_k ~=~ \langle a, b \mid ba^{k-1}b = a^k, ~ ab^{k-1}a = b \rangle.
\end{align*}
This argument can be extended to general $R_k^n$ with $n \geq 2$.
Concretely, the isomorphism $f: S_k \to \mathrm{Ker}(\psi)/\langle \langle x^2 \rangle \rangle$ is defined by
\begin{align*}
    a \mapsto xz^{-1}xy^{-1}zx^{-1}, \quad b \mapsto yx^{-1}.
\end{align*}
Furthermore, they proved the following (\cite[Proposition 2.3]{Kanenobu-Sumi2020-JKTR}):

\begin{proposition}[\cite{Kanenobu-Sumi2020-JKTR}]\label{Result of Kanenobu-Sumi}
    For each $n \geq 2$ and $k \geq 1$, the number of conjugacy classes of non-abelian $\SL(2, \C)$-representations of $S_k$ is $k-1$.
\end{proposition}

Finally, we prove \Cref{Main Theorem} and \Cref{Type of suciu ribbon knots}.

\begin{proof}[Proof of \Cref{Main Theorem}]
    By \Cref{Result of Winker}, $P_2(Q(R_k^n))$ is isomorphic to $\pi_1(\Sigma_2(R_k^n))$.
    By \Cref{Result of Kanenobu-Sumi}, the fundamental groups $\pi_1(\Sigma_2(R_k^n))$, $k \geq 1$, are mutually non-isomorphic.
    Since $P_2(Q)$ is an invariant of the isomorphism class of a quandle $Q$, the knot quandles $Q(R_k^n)$, $k \geq 1$, are mutually non-isomorphic.
\end{proof}

\begin{proof}[Proof of \Cref{Type of suciu ribbon knots}]
    Let $x$ and $y$ be the generators of $G_k$.
    Since $x$ and $y$ are represented by meridional loops, there exist $x_0, y_0 \in Q(R_k^n)$ such that $\eta(x_0) = x$ and $\eta(y_0) = y$, where $\eta: Q(R_k^n) \to G_k$ is the map defined in \Cref{Section: quandle}.
    We prove that $x_0*^i y_0 \neq x_0$ holds for any $i \geq 1$.

    Let $\phi: B_3 \to \GL(2, \Z[t^{\pm 1}])$ be the group homomorphism defined by
    \begin{align*}
        \sigma \mapsto
        \begin{pmatrix}
            0 & 1 & 0 \\
            t & 0 & 0 \\
            0 & 0 & 1
        \end{pmatrix}, \quad
        \tau \mapsto
        \begin{pmatrix}
            1 & 0 & 0 \\
            0 & 0 & 1 \\
            0 & t & 0
        \end{pmatrix}.
    \end{align*}
    This homomorphism $\phi$ is called the \textit{Tong-Yang-Ma representation} \cite{Tong-Yang-Ma1996} of the braid group $B_3$.
    We identify $G_k$ with $B_3$ by \Cref{Isomorphisms of knot groups}.
    Then, $\phi$ induces a group homomorphism $\phi: G_k \to \GL(2, \Z[t^{\pm 1}])$.

    \begin{claim}\label{claim}
        $\phi(y^i)\phi(x) \neq \phi(x)\phi(y^i)$ holds for any $i \geq 1$.
    \end{claim}

    We show \Cref{Type of suciu ribbon knots} using the claim.
    Suppose that there exists $i \geq 1$ such that $x_0*^i y_0 = x_0$.
    Since $\eta(x_0*^i y_0) = y^{-i} x y^{i}  \in G_k$, we have $y^{-i} x y^{i} = x$.
    Then, by the claim, this contradicts that $\phi$ is a group homomorphism.
    Hence, $\mathrm{type}(Q(R_k^n)) = \infty$ holds for $k \geq 1$ and $n \geq 2$.
\end{proof}

\begin{proof}[Proof of \Cref{claim}]
    Let $X = \phi(x)$, $Y = \phi(y)$.
    We show $Y^i X \neq X Y^i$ for each $i \geq 1$ and $k \geq 1$.
    Let $X_0$, $Y_0$, and $U$ be the matrices defined by
    \begin{align*}
        X_0 ~=~ \phi(\tau^{-1} \sigma \tau) =
        \begin{pmatrix}
            0 & 1 & 0 \\
            1 & 0 & 0 \\
            0 & 0 & t
        \end{pmatrix}, \quad
        Y_0 ~=~ \phi(\sigma) =
        \begin{pmatrix}
            1 & 0 & 0 \\
            0 & 0 & 1 \\
            0 & t & 0
        \end{pmatrix}, \quad
        U ~=~ \phi(\tau^{-1} \sigma) =
        \begin{pmatrix}
            0 & 0 & t^{-1} \\
            1 & 0 & 0 \\
            0 & t & 0
        \end{pmatrix}.
    \end{align*}
    By \Cref{Isomorphisms of knot groups}, we have $X = X_0 U^{k-2}$ and $Y = Y_0 U^{k-1}$.
    In addition, $U^3 = I$ is the identify matrix.
    Thus, the claim can be divided into three cases depending on the value of $k \pmod{3}$:
    \[
        \mbox{\textbf{(a)}}~ X = X_0 U   \mbox{ and } Y = Y_0 U^2, \quad
        \mbox{\textbf{(b)}}~ X = X_0 U^2 \mbox{ and } Y = Y_0, \quad
        \mbox{\textbf{(c)}}~ X = X_0     \mbox{ and } Y = Y_0 U. \quad
    \]
    In this note, we only check the case \textbf{(a)}, and the other cases can be shown in a similar way.

    From the assumption, $X$ and $Y$ are explicitly given by
        \begin{align*}
        X =
        \begin{pmatrix}
            0 & 1 & 0 \\
            1 & 0 & 0 \\
            0 & 0 & t
        \end{pmatrix}
        \begin{pmatrix}
            0 & 0 & t^{-1} \\
            1 & 0 & 0 \\
            0 & t & 0
        \end{pmatrix} =
        \begin{pmatrix}
            1 & 0 & 0 \\
            0 & 0 & t^{-1} \\
            0 & t^2 & 0
        \end{pmatrix}, \quad
        Y =
        \begin{pmatrix}
            1 & 0 & 0 \\
            0 & 0 & 1 \\
            0 & t & 0
        \end{pmatrix}
        \begin{pmatrix}
            0 & 1 & 0 \\
            0 & 0 & t^{-1} \\
            t & 0 & 0
        \end{pmatrix} =
        \begin{pmatrix}
            0 & 1 & 0 \\
            t & 0 & 0 \\
            0 & 0 & 1
        \end{pmatrix}.
    \end{align*}
    For an integer $j \geq 0$, we have
    \begin{align*}
        Y^{2j} ~=~
        \begin{pmatrix}
            t^j & 0 & 0 \\
            0 & t^j & 0 \\
            0 & 0 & 1
        \end{pmatrix}, \quad
        Y^{2j+1} ~=~
        \begin{pmatrix}
            0 & t^j & 0 \\
            t^{j+1} & 0 & 0 \\
            0 & 0 & 1
        \end{pmatrix}.
    \end{align*}
    Hence,
    \begin{align*}
        Y^i X ~=~
        \begin{cases}
        \begin{pmatrix}
            t^j & 0 & 0 \\
            0 & 0 & t^{j-1} \\
            0 & t^2 & 0
        \end{pmatrix} & (i = 2j),\\
        \begin{pmatrix}
            0 & 0 & t^{j-1} \\
            t^{j+1} & 0 & 0 \\
            0 & t^2 & 0
        \end{pmatrix} & (i = 2j+1).
        \end{cases} \quad
        X Y^i ~=~
        \begin{cases}
        \begin{pmatrix}
            t^j & 0 & 0 \\
            0 & 0 & t^{-1} \\
            0 & t^{j+2} & 0
        \end{pmatrix} & (i = 2j),\\
        \begin{pmatrix}
            0 & t^j & 0 \\
            0 & 0 & t^{-1} \\
            t^{j+3} & 0 & 0
        \end{pmatrix} & (i = 2j+1).
        \end{cases}
    \end{align*}
    Therefore, $Y^i X \neq X Y^i$ holds for any integer $i \geq 1$.
\end{proof}

%% file: acknowlegements.tex

\medskip
\noindent
\textbf{Acknowledgements.}
This work was supported by JSPS KAKENHI Grant Number 25K23341 and Research Fellowship Promoting International Collaboration, The Mathematical Society of Japan.